\documentclass[11pt]{amsart}

\usepackage{hyperref,verbatim}

\newcommand*{\mailto}[1]{\href{mailto:#1}{\nolinkurl{#1}}}

\newtheorem{theorem}{Theorem}[section]
\newtheorem{proposition}[theorem]{Proposition}
\newtheorem{lemma}[theorem]{Lemma}
\newtheorem{remark}[theorem]{Remark}
\newtheorem{corollary}[theorem]{Corollary}

\newcommand{\qq}{\quad\quad}

\newcommand{\nf}{\infty}

\newcommand{\E}{\mathrm{e}}

\newcommand{\al}{\alpha}
\newcommand{\be}{\beta}

\newcommand{\de}{\delta}
\newcommand{\De}{\Delta}

\newcommand{\ve}{\varepsilon}

\newcommand{\vp}{\varphi}
\newcommand{\om}{\omega}

\newcommand{\rrr}{\mathbf R}

\newcommand{\rn}{\mathbf R^n}

\newcommand{\ccc}{\mathbf C}

\newcommand{\stwo}{\mathbf S^2}

\newcommand{\sn}{\mathbf S^{n-1}}

\newcommand{\cf}{\mathcal F}
\newcommand{\ch}{\mathcal H}
\newcommand{\cs}{\mathcal S}

\newcommand{\intrn}{\int_{\rn}}

\newcommand{\p}{\partial}
\newcommand{\sech}{\mathrm{sech}}

\newcommand{\f}{\frac}
\newcommand{\dint}{\displaystyle\int}

\begin{document}

\title[On radial Fourier transforms]{On Fourier transforms of radial functions and distributions}

\author[L.\ Grafakos]{Loukas Grafakos}
\address{Department of Mathematics\\
University of Missouri\\
Columbia, MO 65211, USA}
\email{\mailto{grafakosl@missouri.edu}}
\urladdr{\url{http://www.math.missouri.edu/~loukas/}}

\author[G.\ Teschl]{Gerald Teschl}
\address{Faculty of Mathematics\\ University of Vienna\\
Nordbergstrasse 15\\ 1090 Wien\\ Austria\\ and International
Erwin Schr\"odinger
Institute for Mathematical Physics\\ Boltzmanngasse 9\\ 1090 Wien\\ Austria}
\email{\mailto{Gerald.Teschl@univie.ac.at}}
\urladdr{\url{http://www.mat.univie.ac.at/~gerald/}}
 
\thanks{J. Fourier Anal. Appl. {\bf 19}, 167--179 (2013)}
\thanks{Grafakos' research was supported by the NSF (USA) under grant DMS 0900946.
Teschl's work was supported by the Austrian Science Fund (FWF) under Grant No.\ Y330}




\begin{abstract}
 We find a formula that relates the Fourier transform of a radial function on $\mathbf{R}^n$ with the Fourier transform of the
same function defined on $\mathbf{R}^{n+2}$. This formula enables one to explicitly calculate the Fourier transform of any
radial function $f(r)$ in any dimension, provided one knows the Fourier transform of the one-dimensional function 
$t\mapsto f(|t|)$ and the 
two-dimensional function $(x_1,x_2)\mapsto f(|(x_1,x_2)|)$. We prove analogous results for radial tempered distributions. 
 \end{abstract}

\maketitle

\section{Introduction}

The Fourier transform of a function $\Phi$ in $L^1(\rn)$ is
defined by the convergent integral
\[
F_{n}(\Phi)(\xi)=\intrn \Phi(x) \E^{-2\pi i x\cdot \xi}\, dx\, .
\] 
If the function $\Phi$ is radial, i.e., $\Phi(x)= \vp(|x|)$ for some function $\vp$ on the line, 
then its Fourier transform is also radial and we use the notation 
\[
 F_{n}(\Phi)(\xi)= \cf_{n}(\vp)(r)\, , 
\]
where $r=|\xi|$. In this article, we will show that there is a relationship between 
$\cf_{n}(\vp)(r)$ and $\cf_{n+2}(\vp)(r)$ as functions of the positive real variable $r$.\linebreak

We have the following result.
\begin{theorem}\label{schwartz}
Let $n\ge 1$. Suppose that $f$ is a function on the real line such that the functions 
 $f(|\cdot |)$ are in $L^1(\rrr^{n+2}) $ and also in $L^1(\rn)$. Then we have 
 \begin{equation}\label{1000-1}
\cf_{n+2}(f)(r) = - \frac{1}{2\pi} \frac{1}{r} \frac{d }{dr} \cf_{n} (f)(r)\qquad r>0 \,. 
\end{equation}
Moreover, the following formula is valid for all even Schwartz functions $\vp$ on the real line:
\begin{equation}\label{55}
 \cf_{n+2} (\vp)(r)= \f{1}{2\pi}\f{1}{r^2} \cf_{n} \Big(s^{-n+1} \f{d}{ds} (\vp(s)s^n) \Big)(r), \qquad r>0 \,.
\end{equation}
\end{theorem}
Using the fact that the Fourier transform is a unitary operator on $L^2(\rn)$ we may extend \eqref{1000-1} to the case where the functions 
 $f(|\cdot |)$ are in $L^2(\rrr^{n+2}) $ and in $L^2(\rn)$. Moreover,
in Section~\ref{distr} we extend \eqref{1000-1} to tempered distributions. Applications are given in the last section.

\begin{corollary}
Let $f(r)$ be a function on $[0,\nf)$ and $k$ some positive integer such the functions $x\to f(|x|)$ are absolutely integrable over $\rn$ for all $n$ with $1 \le n \le 2k+2$. Then we have
\[
\mathcal F_{2k+1}(f)(\rho) = \f{1}{( 2\pi)^k}\sum_{\ell=1}^{k} \frac{(-1)^\ell (2k-\ell-1)!}{2^{k-\ell} (k-\ell)! \, (\ell-1)!} \, \frac{1}{\rho^{2k-\ell}}\, 
\bigg(\frac{d }{d\rho }\bigg)^\ell \mathcal F_{1}(f)(\rho)
\]
and
\[
\mathcal F_{2k+2}(f)(\rho) = \f{1}{(2\pi)^k}\sum_{\ell=1}^{k} \frac{(-1)^\ell (2k-\ell-1)!}{2^{k-\ell} (k-\ell)! \, (\ell-1)!} \,\frac{1}{\rho^{2k-\ell}}\, 
\bigg(\frac{d }{d\rho }\bigg)^\ell \mathcal F_{2}(f)(\rho).
\]
\end{corollary}

The Corollary can be obtained using \eqref{1000-1} by induction on $k$. The simple details are omitted. Again,  absolute  integrability can be replaced by square integrability.

\section{The proof}

The Fourier transform of an integrable radial function $f(|x|)$ on $\rn$ is given by
\begin{align*}
\cf_{n}(f)(|\xi|)=& 2\pi \dint_0^\infty f(s) \Big(\frac{s}{|\xi|}\Big)^{{\f n2}-1} J_{\f n2-1}(2\pi s |\xi| )\, s \, ds\\
=& {(2\pi)^{{{\f n2} }}} \dint_0^\infty f(s) \widetilde{J}_{\f n2-1}(2\pi s |\xi|)\, s^{n-1} ds \, , 
\end{align*}
where $\widetilde{J}_\nu(x)= x^{-\nu} J_\nu(x)$, and $J_\nu$ is the classical Bessel function 
of order $\nu$.
This formula can be found in many textbooks, and
we refer to, e.g., \cite[Sect.~B.5]{grafakos} or \cite[Sect.~IV.1]{stein-weiss} for a proof. 
Moreover, this formula makes sense for all integers $n\ge 1$, even $n=1$, in which case
\[ 
J_{-1/2}(t) = \sqrt{\frac{2}{\pi}} \frac{\cos t}{\sqrt{t}}\, . 
\]
Let us set 
\[
\ch_{\f n2 -1}(f)(r)=
{(2\pi)^{{{\f n2} }}} \dint_0^\infty f(s) \widetilde{J}_{\f n2-1}(2\pi s r)\, s^{n-1} ds \, . 
\]

Then we make use of B.2.(1) in \cite{grafakos}, i.e., the identity
\begin{equation}\label{987}
\f{d}{dr}\widetilde{J}_\nu (r) = -r \widetilde{J}_{\nu+1}(r),
\end{equation}
which is also valid when $\nu=-1/2$, since 
\[
J_{1/2}(t)= \sqrt{\frac{2}{\pi}} \frac{\sin t}{\sqrt{t}}\, .
\]
In view of \eqref{987}, it is straightforward to verify that
\[
- \frac{1}{r} \frac{d}{dr} \ch_{\f n2 -1} (f)(r) =2 \pi \ch_{\f n2} (f)(r) =2 \pi \ch_{\f {n+2}2 -1 } (f)(r)\, ,
\]
provided $f$ is such that interchanging differentiation with the integral defining 
$ \ch_{\f n2 -1}$ 
 is permissible. For this to happen, we need to have that
\[
\int_0^\infty |f(s) | \Big| \frac{d}{dr}\Big( \widetilde{J}_{\frac n2-1} (rs) \Big)\Big|s^{n-1} ds <\infty 
\]
and thus it will be sufficient to have
\begin{equation}\label{99}
\int_0^\infty |f(s)|\, rs^{2 } |\widetilde{J}_{\f n2}(rs)| s^{n-1} ds \le c
\int_0^\infty |f(s) | \frac{rs^{2 }}{(1+rs)^{\f{n+1}2} } s^{n-1} ds <\infty 
\end{equation}
since $|\widetilde{J}_{\f n2}(s)| \le c (1+s)^{-n/2-1/2} $. But since $f(|\cdot |)$ is
in $L^1(\rrr^{n+2})$ we have
\begin{equation}\label{999}
 \int_0^{1/r} |f(s) | s^{n+1}\, ds+ \int_{1/r}^\infty |f(s) | s^{\f{n+1}{2}}\, ds<\infty
\end{equation}
 and this certainly implies \eqref{99} for all $r>0$. 
We conclude \eqref{1000-1} whenever \eqref{999} holds. We note that the appearance of condition \eqref{999} 
is natural as indicated in \cite{samko} (Lemma 25.1).

To prove \eqref{55} we argue as follows. We have
\[
\ch_{\f n2 -1}\Big(r^{-n+1} \f{d}{dr} (\vp(r)r^n) \Big)(r)=
{(2\pi)^{{{\f n2} }}} \dint_0^\infty \f{d}{ds} \big(\vp(s)s^n\big) \, \widetilde{J}_{\f n2-1}(2\pi s r)\, ds 
\]
and integrating by parts the preceding expression becomes
\[
 (2\pi)^{\f n2 +2 }\dint_0^\infty \vp(s)s^n sr^2 \, \widetilde{J}_{\f {n+2}2-1}(2\pi s r)\, ds 
\] 
which is equal to $2\pi r^2 \ch_{\f{n+2}2-1} (\vp)(r)$. This proves \eqref{55}. 
 
\begin{remark}\label{remhankel}
Note that we have
\[
\ch_\nu(f)(r)= \frac{2\pi}{r^\nu} H_\nu(f(s)s^\nu)(2\pi r),
\]
where
\[
H_{\nu}(f)(r) = \int_0^\infty f(s) J_\nu(r s) s\, ds, \qquad \nu \ge -\frac{1}{2},
\]
is the Hankel transform. This of course ties in with the fact that the Hankel transform
also arises naturally as the spectral transformation associated with the radial part of
the Laplacian $-\Delta$; we refer to \cite[Sect.~5]{KST} and the references therein for further
information. Moreover, note that \cite{SP} contains the associated recursion from Theorem~\ref{schwartz}
for the Hankel transform, but only for even Schwartz functions. This recursion was rediscovered in
connection with the radial Fourier transform in \cite{SW} for the case of Schwartz functions. See also
\cite{LT} for related results.

A transference theorem for radial multipliers 
which exploits the connection between the Fourier transform of radial functions on 
$\mathbf R^n$ and $\mathbf R^{n+2}$ was obtained in \cite{CW}. This multiplier theorem is based
on an identity dual to \eqref{987}.
\end{remark}

\section{Radial distributions}

We denote by $\cs (\rn)$ the space of Schwartz functions on $\rn$ and by $\cs' (\rn)$ the space of tempered distributions on $\rn$. 
A Schwartz function is called radial if for all orthogonal transformations $A\in O(n)$ (that is, for all rotations on $\rn$) we have
\[
\vp = \vp\circ A\,.
\]
We denote the set of all radial Schwartz functions by $\cs_{rad}(\rn)$. For further background on radial distributions we refer to Treves \cite[Lect.~5]{treves}.
Observe that in the one-dimensional case the radial Schwartz functions
are precisely the even Schwartz functions, that is:
\[
\cs_{rad}(\rrr) =\cs_{even}(\rrr) = \{ \vp \in \cs(\rrr):\,\, \vp(x)=\vp(-x)\}.
\]
 Similarly, a distribution $u\in \cs'(\rn)$ is called radial if for all orthogonal transformations $A\in O(n)$ we have
\[
u = u\circ A\, .
\]
This means that 
\[
\langle \, u,\vp \, \rangle \, = \, \langle\, u ,\vp\circ A \, \rangle
\]
for all Schwartz functions $\vp$ on $\rn$. We denote by $\cs'_{rad}(\rn)$ the space of all 
radial tempered distributions on $\rn$. We also denote by $\sn$ the $(n-1)$-dimensional 
unit sphere on $\rn$ and by $\om_{n-1}$ its surface area. 

Given a general, non necessarily radial, Schwartz function there is a natural homomorphism 
\[
\cs(\rn) \to \cs_{rad}(\rrr), \quad \vp(x) \mapsto \vp^o(r)= \f{1}{\om_{n-1}}\int_{\sn} \vp(r \theta)\, d\theta
\]
with the understanding that when $n=1$, then $\vp^o(x)=\frac{1}{2}(\vp(x)+\vp(-x))$. 
Conversely, given an even Schwartz function on $\rrr$ we can define a corresponding radial Schwartz function
via
\[
\cs_{rad}(\rrr) \to \cs_{rad}(\rn), \quad \vp(r) \mapsto \vp^O(x)=\vp(|x|).
\]
The map $\vp\mapsto \vp^O$ is a homomorphism; the proof of this fact is omitted since a stronger statement
is proved at the end of this section. Both facts require the following lemma:

\begin{lemma}\label{whitney}
Suppose that $f$ is a smooth even function on $\rrr$. Then there is a smooth function $g$ on the real line such that 
\[
f(x) = g(x^2)
\]
for all $x\in \rrr$. Moreover, one has for $t\ge 0$
\begin{equation}\label{estg}
|g^{(k)}(t) | \le C(k) \sup_{0\le s \le \sqrt{t }} |f^{(2k)}(s)|.
\end{equation}
\end{lemma}

\begin{proof}
By Whitney's theorem \cite{W}, there is a smooth function $g$ on the real line such that 
\[
f(t) = g(t^2)
\]
for all real $t$.

To see the last assertion we use the following representation of the remainder in Taylor's theorem:
\begin{align*}
\f{g^{(k)}(t^2)}{k!} &=  (2t)^{-2k+1} k {2k \choose k }\int_0^t (t^2-s^2)^{k-1} \f{f^{(2k)} (s) }{(2k)!}\, ds \\
&=  2^{-2k } k {2k \choose k }\int_0^1 (1-s^2)^{k-1} \f{f^{(2k)} (st) }{(2k)!}\, ds
\end{align*}
from which one easily derives \eqref{estg}. This yields in particular that 
\[
\f{g^{(k)}(0)}{k!} =\f{f^{(2k)} (0) }{(2k)!}
\]
since
\[
2^{-2k } k  {2k \choose k }\int_0^1 (1-s^2)^{k-1} ds =2^{-2k } k  {2k \choose k } \frac{\Gamma(k)\Gamma(1/2)}{\Gamma(k+1/2)}  = 1.  
\]
\end{proof}

The composition $\vp\mapsto (\vp^o)^O=\vp^{rad}$ gives rise to a homomorphism from 
$\cs(\rn) \to \cs_{rad}(\rn)$ which reduces to
the identity map on radial Schwarz functions. 
In particular, the map $\vp\mapsto\vp^o$ defines a one-to-one correspondence between
radial Schwartz functions on $\rn$ and even Schwartz functions on the real line. 
Moreover, $\vp$ is radial if and only if $\vp=\vp^{rad}$.

\begin{proposition}
For $ u \in \cs'_{rad}(\rn)$ and $ \vp\in \cs(\rn) $ we have
\[
\langle u,\vp \rangle = \langle u ,\vp^{rad} \rangle\, . 
\]
\end{proposition}

\begin{proof}
By a simple change of variables the formula holds for any $u$ which is a polynomially bounded locally integrable function.
Next we fix a tempered distribution $u$ on $\rn$ and we consider a radial Schwartz function $\psi$ with integral $1$ and we set $\psi_\ve(x) = \ve^{-n} \psi(x/\ve)$. Then we notice that the convolution of 
$\psi_\ve * u$ converges to $u$ in $\cs'(\rn)$ as $\ve\to 0$. Hence, since the claim holds if $u$ is replaced by $\psi_\ve * u$
by the first observation, it remains true in the limit $\ve\to 0$.
\end{proof}

In particular, note that a radial distribution is uniquely determined by its action on radial Schwartz functions.
Furthermore, given a distribution $u\in \cs'(\rn)$ we can define a radial distribution $u^{rad}\in \cs'_{rad}(\rn)$ via
\[
\langle u^{rad},\vp \rangle := \langle u ,\vp^{rad} \rangle.
\]
Moreover, $u$ is radial if and only if $u=u^{rad}$.

For $n\in \mathbf Z^+$ we denote by $\mathcal R_n = r^{n-1}\mathcal S_{even}(\rrr)$ the space of functions of the
form $\psi(r)r^{n-1}$, where $\psi$ is an even Schwartz function on the line. This space inherits the topology of 
$S(\rrr)$ and its dual space is denoted by $\mathcal R_n'$.
Two distributions $w_1, w_2 \in\mathcal S'(\rrr)$ are equal in the space $\mathcal R_n'$ if for all 
even Schwartz functions $\psi$ on the line we have:
\[
\langle w_1, r^{n-1} \psi(r) \rangle = \langle w_2, r^{n-1} \psi(r) \rangle \, . 
\]
Note that in dimension $n\ge 2 $ we have that all distributions
of order $n-2$ supported at the origin equal
the zero distribution
 in the space $\mathcal R_n'$. Thus two radial distributions $w_1$ and $w_2$ are equal in 
$\mathcal R_n'$ whenever $w_1-w_2$ is a 
sum of derivatives of the Dirac mass at the origin of 
order at most $n-2$.

One may build radial distributions on $\rn$ starting from
distributions in $\mathcal R_n'$. 
Indeed, given $u_\diamond$ in $\mathcal R_n'$ 
and $\vp$ in $\cs(\rn)$ we define a radial distribution $u$ by setting 
\[
\langle\, u,\vp\,\rangle\, := \, \frac{\om_{n-1}}{2} \langle\, u_\diamond , \vp^o (r) r^{n-1}\, \rangle 
\]
The converse is the content of the following proposition.

\begin{proposition}
The map $\mathcal R_n \to \mathcal S_{rad}(\rn)$, $\psi(r) r^{n-1}\mapsto \psi^O(x)$ is
a homeomorphism and hence for every radial distribution $u$ we can define $u_\diamond$
in $\mathcal R_n'$ via
\[
\langle\, u_\diamond , \psi(r)r^{n-1} \rangle\, := \, \f{2}{\om_{n-1}} \langle\, u,\psi^O\,\rangle.
\]
\end{proposition}

\begin{proof}
It suffices to show the first claim. To this end we will show that for all 
multiindices $\al$ and $\be$ we have
\[
\sup_{x\in \rn}|x^\al \p^\be_x ( \psi (|x|)) | \le \sum_{0\le \ell,m\le 4(|\be|+|\al|+n)} \sup_{r>0} |r^m \Big(\f{d}{dr}\Big)^\ell (r^{n-1}\psi(r)) | \, .
\]
First we consider the case $|x|\le 1$. Setting $r=|x|\le 1$ we have
\begin{align*}
| x^\al \p^\be_x ( \psi (|x|)) | & \le C_{ \be} |x|^{|\al|}
 \sum_{k=0}^{|\be|}|x|^k |g^{(k)}(|x|^2)| 
 = C_{ \be} 
 \sum_{k=0}^{|\be|}| r^{k+|\al|} g^{(k)}(r^2)| \\ 
 &\le C_{ \be} 
 \sum_{k=0}^{|\be|}| g^{(k)}(r^2)| 
 \le C_{ \be} 
 \sum_{k=0}^{|\be|} C(k) \sup_{0<s<r} |\psi^{(2k)}(s)| \, , 
\end{align*}
using Lemma~\ref{whitney} with $\psi(t)=g(t^2)$. 

We will make use of the inequality
\begin{equation}\label{hyref}
| \psi (s) | \le \sup_{0<t<s} \Big| \Big(\f{d}{dt}\Big)^{M} (t^{M}\psi (t) )(s) \Big|
\end{equation}
which follows by applying the fundamental theorem of calculus $M$ times and of the
identity:
\begin{equation}\label{hyref2}
s^M \f{d^m\psi}{ds^m} (s) = \sum_{\ell=0}^m (-1)^\ell \ell! {m \choose \ell} {M \choose \ell} 
\Big(\f{d}{ds}\Big)^{m-\ell} (s^{M-\ell}\psi (s) ) 
\end{equation}
which is valid for $M\ge m$ and is easily proved by induction. 

Applying \eqref{hyref} to $\psi^{(2k)}(s)$ we obtain 
\begin{equation}\label{hyref3}
| \psi^{(2k)}(s) | \le \sup_{0<t<s} \Big| \Big(\f{d}{dt}\Big)^{M} (t^{M} \psi^{(2k)} (t) )(s) \Big|
\end{equation}
and using \eqref{hyref2} for $s^{M} \psi^{(2k)} (s)$ with $M=2|\be|+n-1$ and $m=2k$ we deduce that 
$| \psi^{(2k)}(s) | $ is pointwise bounded by a sum of derivatives of terms $s^{n-1}\psi(s)$ multiplied by 
powers of $s$. It follows that $\sup_{s>0} |\psi^{(2k)}(s)|$ is controlled by a finite sum of Schwartz 
seminorms of the function $s^{n-1}\psi(s)$. 

The case $|x|\ge 1$ is easier since when $|\be|\neq 0$
\[
| \p^\be_x ( \psi(|x|))| \le \sum_{j=1}^{|\be|}| \psi{(j)}(|x|)| \f{C_{j,\be}}{|x|^{|\be|-j}}\, , 
\]
and taking $M=\max ( |\al|, |\be|+n-1)$ we have 
\begin{equation}\label{uu1555}
\sup_{|x|\ge 1} \big|x^\al \p^\be_x (\psi(|x|)) \big| 
\le C_{\be} \sum_{j=1}^{|\be|}\sup_{s\ge 1} \big\{ s^{M} |\psi^{(j)}(s)|\big\}\, , 
\end{equation}
which is certainly controlled by a finite sum of Schwartz seminorms of $s^{n-1}\psi(s)$ in view of \eqref{hyref2}.
\end{proof}

Note that if $u$ is given by a function $f(x)$, then $u_\diamond$ is given by the function $f^o(x)$. We also remark
that the map $\frac{1}{r} \frac{d}{dr}$ is a homomorphism from $\mathcal R_n'$ to $\mathcal R_{n+1}'$
defined as the dual map of $- \frac{d}{dr}\frac{1}{r}$.

A related approach defining $u_\diamond$ for a given distribution $u$ supported in $\rn\setminus\{0\}$ can be found in \cite{szmydt}. Our approach does not impose restrictions on the support of the distribution.

\section{The extension to tempered distributions}
\label{distr}

Let $u$ be a radial distribution on $\mathbf R^k$ and let $F_{k}(u)$ be the $k$-dimensional 
Fourier transform of $u$. 

\begin{theorem} \label{distribution}
Given an even tempered distribution $v_0$ on the real line, define radial distributions $v_n$ on $\rn$ and $v_{n+2}$ on $\mathbf R^{n+2}$ via the identities
\begin{equation}\label{665}
\langle v_n, \vp \rangle = \langle v_0, \tfrac{1}{2}\omega_{n-1} r^{n-1} \vp^o \rangle
\end{equation}
for all radial Schwartz functions $\vp(x)$ = $\vp^o(|x|)$ on $\rn$ and 
\[
\langle v_{n+2}, \vp \rangle = \langle v_0, \tfrac{1}{2} \omega_{n+1} r^{n+1} \vp^o \rangle
\]
for all radial Schwartz functions $\vp(x)$ = $\vp^o(|x|)$ on $\mathbf R^{n+2}$. 

Let $u^n=F_n(v_n) $ and $u^{n+2}=F_{n+2}(v_{n+2})$. Then the identity
\begin{equation}\label{newformula}
 -\f{1}{2\pi r}\f{d}{dr} u_\diamond^n = u_\diamond^{n+2} 
\end{equation}
holds on $\mathcal R_{n+2}'$. 
\end{theorem}

\begin{proof}
We denote by $\langle \cdot , \cdot \rangle_n$ the action of the distribution on a function in dimension $n$. 
Let $\psi(r)$ be an even Schwartz function on the real line. Then we need to show that 
\begin{equation}\label{tp1}
\Big\langle -\f{1}{2\pi r}\f{d}{dr} u_\diamond^n, \om_{n+1} r^{n+1} \psi(r) \Big\rangle_1 =
\big\langle u_\diamond^{n+2} , \om_{n+1} r^{n+1} \psi(r) \big\rangle_1 \, . 
\end{equation}
This is equivalent to showing that 
\begin{equation}\label{tp2}
\f{1}{2\pi } \big\langle u_\diamond^n , \om_{n+1} (r^{n } \psi(r))' \big\rangle_1 =
\big\langle u_\diamond^{n+2} , \om_{n+1} r^{n+1} \psi(r) \big\rangle_1 \, . 
\end{equation} 
We introduce the even Schwartz function $\eta(r) = r^{-n+1} (r^{n } \psi(r))'=n\psi(r) +r\psi'(r)$ on the real line and functions $\eta^O$ on $\rn$ and $\psi^O$ on $\mathbf R^{n+2}$ by setting 
\[
\psi^O (x) = \psi (|x|) \qq \qq\qq \eta^O(y) = \eta (|y|) 
\]
for $y\in \rn$ and $x\in \mathbf R^{n+2}$. Then \eqref{tp2} is equivalent to 
 \begin{equation}\label{tp3}
\f{1}{2\pi }\f{\om_{n+1}}{\om_{n-1}} \big\langle u_\diamond^n , \om_{n-1} r^{n-1} \eta(r) \big\rangle_1 =
 \big\langle u_\diamond^{n+2} , \om_{n+1} r^{n+1} \psi(r) \big\rangle_1
\end{equation} 
which is in turn equivalent to 
 \begin{equation}\label{tp4}
\f{1}{2\pi }\f{\om_{n+1}}{\om_{n-1}} \big\langle F_n(v_n) , \eta^O \big\rangle_n =
 \big\langle F_{n+2}(v_{n+2}) , \psi^O \big\rangle_{n+2} 
\end{equation} 
and also to 
 \begin{equation}\label{tp5}
\f{1}{2\pi }\f{\om_{n+1}}{\om_{n-1}} \big\langle v_n , F_n(\eta^O) \big\rangle_n =
 \big\langle v_{n+2}, F_{n+2}(\psi^O) \big\rangle_{n+2} \, . 
\end{equation} 
 
We now switch to dimension one by writing \eqref{tp5} equivalently as 
 \begin{equation}\label{tp6}
\f{1}{2\pi }\f{\om_{n+1}}{\om_{n-1}} \big\langle v_0 , \om_{n-1}r^{n-1} \mathcal F_n(\eta )(r) \big\rangle_1 =
 \big\langle v_{0}, \om_{n+1} r^{n+1}\mathcal F_{n+2}(\psi ) (r) \big\rangle_{1} \, . 
\end{equation} 
But this identity holds if
\[
\f{1}{2\pi} \mathcal F_n(\eta )(r) = r^2 \mathcal F_{n+2}(\psi ) (r)\, , 
\]
which is valid as a restatement of \eqref{55}; recall that $\eta(r) = r^{-n+1} \f{d}{dr} (r^{n } \psi(r) ).$
This proves \eqref{tp1}. 
\end{proof}
 
It is straightforward to check that for polynomially bounded smooth functions all operations
coincide with the usual ones. We end this section with a few more illustrative examples.
Let $\de_n$ be the Dirac mass on $\rn$.

\noindent {\bf Examples:}

a) Let $ v_n = \delta_n$. One can see that 
\[
v_0 = \f{2\, (-1)^{n-1} }{ \omega_{n-1} (n-1)! } \Big(\f{d } {dr}\Big)^{(n-1)}(\delta_1)
\]
satisfies \eqref{665}. 
Acting $v_0$ on $r^{n+1}\vp^o(r)$ yields that 
$ v_{n+2} = 0$ and thus $u_\diamond^{n+2}=0$. Also $u_\diamond^n=1$; so both sides of \eqref{newformula} are
equal to zero. 

b) Let $v_{n+2}=\delta_{n+2}$. Then 
\[
v_0 = \f{2\, (-1)^{n+1} }{ \omega_{n+1} (n+1)! } \Big(\f{d } {dr}\Big)^{(n+1)}(\delta_1)\, . 
\]
Let $\De=\p_1^2+\cdots + \p_n^2$ be the Laplacian. We claim that the distribution 
 \begin{equation}\label{oopp}
 v_n = \f{ \omega_{n-1}}{ \omega_{n+1}} \f{1}{ n}\Delta(\delta_n) 
\end{equation}
 satisfies \eqref{665}. 
 Then $u_\diamond^{n+2}=1$ and also $u_\diamond^n = -r^2 (2\pi )^2 \omega_{n-1} / (2 n \omega_{n+1})$. Thus 
 \eqref{newformula} is valid since $2\pi \om_{n-1}=n\om_{n+1}$. 
 
 It remains to prove
 that the distribution $v_n$ in \eqref{oopp} satisfies \eqref{665}. 
 For $\vp(x) = \vp^o(|x|)$ in $\mathcal S(\rn)$ we have
 \begin{equation}\label{oopp2}
\big\langle v_n, \vp \big\rangle = \langle v_0, \om_{n-1} r^{n-1} \vp^o(r) \big\rangle = 
\f{\om_{n-1}}{\om_{n+1}} \f{2}{(n+1)!} \langle \de_1 , (r^{n-1} \vp^o(r))^{(n-1)}\big\rangle 
\end{equation}
and one notices that the $(n\!-\!1)$st derivative of $r^{n-1} \vp^o(r)$ evaluated at zero is equal to 
$\f12 (n+1)! (\vp^o)''(0)$. To compute the value of this derivative we use Lemma \ref{whitney} to 
write $\vp(x)= \vp^o(|x|) = g(|x|^2)$ where $g'(0)= \f12 (\vp^o)''(0)$. It follows that 
$ g'(0) = \f 1{2n} \De (\vp)(0)$. Combining these 
observations yields that the expression in \eqref{oopp2} is equal to 
\[
\f{\om_{n-1}}{\om_{n+1}} \f{1}{n} \De (\vp)(0) =\Big\langle 
\f{\om_{n-1}}{\om_{n+1}} \f{1}{n} \De (\de_n), \vp \Big\rangle \, , 
\]
which proves the claim.

\begin{remark}
As pointed out in Remark~\ref{remhankel}, the action of the Fourier transform on the associated function on the reals $\vp^o$
is given by the Hankel transform. In particular, the results in this section also give a natural extension of the Hankel transform
(for half-integer order) to distributions. Of course this coincides with the usual approach, see \cite{SP, Z, zemanian} and the
references therein. To this end observe that the space $F$ used in \cite{SP} is precisely the set of functions on $[0,\infty)$ which
extend to an even Schwartz function on $\rrr$.
\end{remark}

 \section{Applications}

We begin with a simple example.
In dimension one we have that the Fourier transform 
of $\textup{sech}(\pi |x|) $ is $\textup{sech}(\pi |\xi|) $. It follows from 
\eqref{1000-1} that in dimension three we have
\[
F_3( \sech(\pi |x|) ) (\xi) = \frac{1}{2|\xi|} \sech(\pi |\xi |) \tanh(\pi |\xi |) \, . 
\]
since 
\[
 \frac{d}{dr} \frac{2}{\E^{\pi r}+\E^{-\pi r} } = - 2\pi \frac{\E^{\pi r}-\E^{-\pi r} }{(\E^{\pi r}+\E^{-\pi r} )^2} =-2\pi \frac 12 \sech(\pi r) \tanh(\pi r)
\]
Continuing this process, one can explicitly calculate the Fourier transform of $\sech(\pi |x|)$ in all odd dimensions. 
 
More sophisticated applications of our formulas appear in computations of functions of the Laplacian $-\Delta$, which arise in numerous applications.
For example, in quantum mechanics the Laplacian $-\Delta$ arises as the free Schr\"odinger operator (cf., e.g., \cite{reedsimon2}, \cite{teschl})
and functions $f(-\Delta)$ are defined via the spectral theorem by
\[
f(-\Delta) \vp = K * \vp, \qquad \vp\in\cs(\rn),
\]
where $K$ is the tempered distribution given by the inverse Fourier transform of the radial function $f(4\pi^2 |\xi|^2)$, which is assumed polynomially bounded. Knowledge of the inverse Fourier transform of $f(4\pi^2 |\xi|^2)$, for $\xi\in \rrr$ and $\xi\in \rrr^2$, yields explicit formulas for the kernel $K$ of $f(-\Delta)$ in all dimensions. 

An important application is the explicit calculation of the $n$-dimensional kernel $G_n(x)$ for the resolvent associated with the function $f(r)=(r-z)^{-1}$, $z\in\ccc\backslash [0,\infty)$. In the one-dimensional case, an easy computation shows that
\[
G_1(x)= \frac{1}{2\sqrt{- z}} \E^{-\sqrt{- z}\, |x|}.
\]
Hence, by the $L^2$ version of Theorem~\ref{schwartz} (cf.\ the discussion right after Theorem~\ref{schwartz}) the three-dimensional kernel is given by
\[
G_3(x)= -\frac{1}{2\pi r} \frac{d}{dr} G_1(r)\Big|_{r=|x|} = \frac{1}{4\pi |x|} \E^{-\sqrt{- z}\, |x|}.
\]
The computation of $G_5(x), G_7(x),\dots$ requires Theorem~\ref{distribution} since
the assumptions of Theorem~\ref{schwartz} are no longer satisfied. 
For instance, Theorem~\ref{distribution} gives 
\[
G_5(x) = \frac{1+|x| \sqrt{- z}}{8\pi^2 |x|^3} \E^{-\sqrt{- z}\, |x|} .
\]

Another interesting situation where our theorem is useful are the spectral projections associated with the function $f(r)=\chi_{[0,E]}(r)$, $E>0$. Again in the 
one-dimensional case the kernel 
for the resolvent can be easily computed and found to be 
\[
P_1(x)=\frac{\sin(x\sqrt{E})}{\pi x}.
\]
Thus 
by Theorem~\ref{schwartz} the three-dimensional kernel is given by
\[
P_3(x)= -\frac{1}{2\pi r} \frac{d}{dr} P_1(r)\Big|_{r=|x|} = \frac{\sin(|x|\sqrt{E}) - |x|\sqrt{E}\cos(|x|\sqrt{E})}{2\pi^2 |x|^3}.
\]

Finally, the Fourier transform is a crucial tool in solving constant coefficient linear partial differential equations (cf., e.g, \cite{evans}).
Using the above trick one can of course derive the fundamental solution for the heat (or Schr\"odinger) equation in three dimensions
from the one-dimensional one. However, since the three-dimensional case is no more difficult than the one-dimensional case we
 rather turn to the Cauchy problem for the wave equation
\[
u_{tt} - \Delta u =0, \qquad u(0,x)= \psi(x), \quad u_t(0,x)=\vp(x),
\]
in $\rn$, whose solution is given by
\[
u(t,x) = \cos(t \sqrt{-\Delta}) \psi(x) + \frac{\sin(t \sqrt{-\Delta})}{\sqrt{-\Delta}} \vp(x).
\]
Since the first term can be obtained by differentiating the second (with respect to $t$) it suffices to look only at the second and assume $\psi=0$.
Moreover, since the Fourier transform of $f(x)=\frac{\sin(a \pi x)}{a \pi x}$ is $F_1(f)(\xi)=|a|^{-1} \chi_{[-1/2,1/2]}(\xi/a)$, we obtain
\[
u(t,x) = \int_\rrr \frac{1}{2} \chi_{[-t,t]}(x-y) \vp(y) dy,
\]
which is of course just d'Alembert's formula. In order to apply Theorem~\ref{distribution} we use $v_0(r)= \frac{\sin(t r)}{r}$ such
that $u^1=F_1^{-1}(v_1)$ as well as $u^1_\diamond$ are associated with the function $\frac{1}{2} \chi_{[-t,t]}(x)$.
Hence by Theorem~\ref{distribution}
\[
\langle F_3^{-1}(v_3), \vp\rangle =
\frac{\om_2}{2} \left\langle -\frac{1}{2\pi r} \frac{d}{dr} \frac{1}{2} \chi_{[-t,t]}(r), r^2 \vp^o(r) \right\rangle =
\frac{\om_2}{4\pi} t \vp^o(t)
\]
and we obtain Kirchhoff's formula
\[
u(t,x) = \frac{t}{4\pi} \int_{\stwo} \vp(x- t \theta) d\theta.
\]

\medskip\noindent{\bf Acknowledgement.}
The authors thank Tony Carbery, Hans Georg Feichtinger, Tom H. Koornwinder, Michael Kunzinger, Elijah Liflyand,
Michael Oberguggenberger, Norbert Ortner, and Andreas Seeger for helpful discussions and hints with respect to the literature.

\end{document}